\documentclass[12pt,a4paper]{article}

\usepackage[font=small,margin=3cm]{caption}

\usepackage{authblk}
\usepackage[margin=2.2cm]{geometry}
\usepackage{t1enc}
\usepackage[utf8]{inputenc}
\usepackage{amsmath,amsthm,amssymb}
\usepackage{graphicx}
\usepackage{enumerate}
\usepackage{hyperref}
\usepackage{bm}
\usepackage{comment}
\usepackage{amsfonts}
\usepackage{graphicx,caption}
\usepackage{bm}
\usepackage{amsmath, amsthm, amssymb}
\usepackage{graphicx}
\usepackage{hyperref}
\usepackage{relsize}
\usepackage{blkarray}
\usepackage{tabstackengine}
\stackMath

\usepackage[table]{xcolor}

\usepackage{algpseudocode}

\usepackage{bbm}

\theoremstyle{plain}
\usepackage{amsthm}
\makeatletter
\newcommand{\newreptheorem}[2]{\newtheorem*{rep@#1}{\rep@title}\newenvironment{rep#1}[1]{\def\rep@title{#2 \ref*{##1}}\begin{rep@#1}}{\end{rep@#1}}}
\makeatother

\newtheorem{theorem}{Theorem}
\newtheorem*{theorem-non}{Theorem}
\newtheorem*{non-lemma}{Lemma}
\newtheorem{lemma}[theorem]{Lemma}
\newreptheorem{lemma}{Lemma}

\theoremstyle{definition}

\newcommand{\V}[2]{V^{#1}_{#2}}
\newcommand{\U}[2]{U^{#1}_{#2}}
\newcommand{\W}[2]{W^{#1}_{#2}}

\DeclareMathOperator{\Sig}{Sig}
\DeclareMathOperator{\rang}{rank}
\DeclareMathOperator{\MC}{MC}

\DeclareMathOperator{\cok}{cok}

\DeclareMathOperator{\Hom}{Hom}

\begin{document}

\title{The fluctuations of the mod p rank of triangular matrices}
\author{Andr\'as M\'esz\'aros}
\date{}
\affil{HUN-REN Alfr\'ed R\'enyi Institute of Mathematics}
\maketitle
\begin{abstract}
We consider random lower triangular matrices such that the entries on and below the diagonal are i.i.d. copies of some $\mathbb{Z}$-valued random variable. We prove that the Sylow $p$-subgroups of the cokernels of these matrices have the same constant order fluctuations as that of the matrix products studied by Nguyen and Van Peski.  

As a special case, we can describe the limiting fluctuations of the rank of lower triangular matrices over $\mathbb{F}_p$ with i.i.d. random entries on and below the diagonal.   
\end{abstract}

\section{Introduction}
Let $p$ be a prime, and let $\xi$ be a $\mathbb{Z}$-valued random variable such that 
\begin{equation}\label{xicond}0<\mathbb{P}(\xi\text{ is divisible by }p)<1.\end{equation}

Let $L_n$ be an $n\times n$ random lower triangular matrix such that the entries on and below the diagonal are i.i.d. copies of $\xi$. The cokernel of $L_n$ is defined as the random abelian group $\cok(L_n)=\mathbb{Z}^n/\mathbb{Z}^n L_n$. Let $\Gamma_n$ be the direct sum of the Sylow $p$-subgroup of the torsion of $\cok(L_n)$ and the free part of $\cok(L_n)$.

\begin{theorem}\label{mainthm}
    Let $n_1<n_2<\cdots$ be a sequence of positive integers such that the fractional parts $\{-\log_p(n_j)\}$ converge to $\zeta$. Let $\chi=\chi_0 p^{-\zeta}/(p-1)$, where $\chi_0$ is defined in Lemma~\ref{lemmachi0exists} below.
 
 Then, for any $d\ge 1$, 
 \[\left(\rang(p^{i-1} \Gamma_{n_j})-\lfloor \log_p(n_j)+\zeta \rceil \right)_{i=1}^{d}\]
 converges to the $\mathbb{Z}^d$-valued random variable $\mathcal{L}_{d,p^{-1},\chi}$ in distribution as $j\to\infty$.
\end{theorem}

In the theorem above, $\lfloor x \rceil$ stands for the integer closest to $x$, that is, $\lfloor x \rceil=\lfloor x+0.5 \rfloor$.

The random variable $\mathcal{L}_{d,p^{-1},\chi}$ can be defined as the single time marginal of an interacting particle system called the \emph{reflecting Poisson sea}~\cite{van2023local,van2023reflecting}. Nguyen and Van Peski~\cite{nguyen2024rank} proved that this random variable can also be characterized by certain types of exponential moments. Moreover, the convergence to $\mathcal{L}_{d,p^{-1},\chi}$ can be established using the so-called \emph{rescaled moment method}. See Section~\ref{sectionrescaled} for more details. Our proof relies on this method and discrete Fourier analysis. 

In the theory of cokernels of random matrices, there is a widespread \emph{universality phenomenon}. For a large class of random matrices, the Sylow $p$-subgroups of their cokernels converge to the \emph{Cohen-Lenstra} distribution or some closely related variant of it \cite{wood2017distribution,wood2019random,nguyen2022random,nguyen2025local,nguyen2024universality,lee2023universality,hodges,gorokhovsky2024time,meszaros2024universal,meszaros2020distribution,meszaros2023cohen,meszaros2024Zpband,lee2025distribution,kang2024random,clancy2015cohen,clancy2015note,friedman1987distribution}. These distributions are also conjectured to appear in the context of arithmetic statistics~\cite{cohen2006heuristics,as1,as2,as3,as4} and random simplicial complexes~\cite{kahle2020cohen,kahle2022topology,meszaros20242}. 

Recently, a new universality class of random matrices was discovered, which shows a different type of behavior. For matrices in this universality class, the size of the Sylow $p$-subgroup of the cokernel goes to infinity as the size of the matrix grows, but the fluctuations of the cokernel remain constant order. The first such examples were found by Van Peski in the integrable setting~\cite{van2023local,van2023reflecting}. A much richer family of examples were found by Nguyen and Van Peski~\cite{nguyen2024rank}. They considered matrices of the form $M_n=A_1A_2\cdots A_k$, where $A_1,\dots,A_k$ are independent $n\times n$ random integral matrices with i.i.d. non-degenerate entries. Under the assumption that $k$ goes to infinity together with $n$, but not too fast, they proved that the fluctuations of the Sylow {$p$-subgroups} of $\cok(M_n)$ are close to $\mathcal{L}_{d,p^{-1},\chi}$, where $\chi$ only depends on $\{-\log_p(k)\}$ but not on the distribution of the entries. The author considered block lower triangular/block bidiagonal matrices with i.i.d. entries in the non-zero blocks, and proved that the cokernels of these matrices show the same limiting behavior as matrix products provided that the sizes of the blocks tend to infinity~\cite{meszaros2024universal}, see also~\cite{meszaros2025rank}. The goal of our paper is to show that a similar result also holds when the blocks are $1\times 1$ matrices, that is, when we have a lower triangular matrix with i.i.d. entries on and below the diagonal. Thus, Theorem~\ref{mainthm} provides new natural examples in the universality class of matrix products. 

The limiting distribution of the fluctuations of $\cok(L_n)$ depends on $\{-\log_p(n)\}$ and also on the distribution of~$\xi$ through the parameter $\chi_0$ defined in Lemma~\ref{lemmachi0exists} below. Therefore, Theorem~\ref{mainthm} gives a weak form of universality, where the limiting object depends slightly on the distribution of the entries. This is quite unusual in the theory of cokernels of random matrices, where the limiting behavior usually does not depend on the distribution of the entries at all \cite{wood2017distribution,wood2019random,nguyen2022random,nguyen2025local,nguyen2024universality,lee2023universality,hodges,gorokhovsky2024time}.  The examples where the limiting behavior does not depend on the distribution of the entries also include the above mentioned matrix products~\cite{nguyen2024rank} and block lower triangular/block bidiagonal matrices~\cite{meszaros2024universal}. 
\begin{lemma}\label{lemmachi0exists}
    The limit
    \[\chi_0=\lim_{n\to\infty} \mathbb{E}|\{v\in\mathbb{F}_p^n\,:\,v_1\neq 0, L_nv=0\}|\]
    exists. Moreover, $0<\chi_0<\infty$.
\end{lemma}

Note that $\chi_0$ only depends on the distribution of $\xi$ mod $p$.

Next, we give another description of the quantities $\rang(p^{i-1} \Gamma_n)$. Let $G_p$ be the Sylow $p$-subgroup of the torsion part of $\cok(L_n)$. Then $G_p$ can be uniquely written as
\[G_p=\bigoplus_{i=1}^r \mathbb{Z}/p^{\lambda_i}\mathbb{Z}\]
for some partition $\lambda=(\lambda_1,\lambda_2,\dots,\lambda_r)$, where $\lambda_1\ge \lambda_2\ge\cdots\ge \lambda_r>0.$

Let $\lambda'$ be the conjugate partition of $\lambda$, that is, $\lambda'_i=|\{j:\lambda_j\ge i\}|$. Moreover, let $b$ be the rank of the free part of $\cok(L_n)$. Then
\[\rang(p^{i-1} \Gamma_n)=\lambda'_i+b.\]

Also note that $\rang(\Gamma_n)$ is just the mod $p$ corank of $L_n$. In other words, if $\overline{L}_n$ is the matrix over $\mathbb{F}_p$ obtained from $L_n$ by considering the mod $p$ remainder of each entry, then $\rang(\Gamma_n)=\dim\ker \overline{L}_n$. Thus, in the special case of $d=1$, Theorem~\ref{mainthm} describes the limiting fluctuations of the mod $p$ corank of lower triangular matrices. Note that the special case of this statement when $\xi$ is uniformly distributed mod $p$ was established in~\cite{van2024rank}. See also~\cite{kirillov1995variations,borodin1995limit,borodin1999law} for other aspects of uniform random lower triangular matrices over finite fields. For the distribution of $\mathcal{L}_{1,p^{-1},\chi}$ the following formula was given in \cite{nguyen2024rank}:
\[\mathbb{P}(\mathcal{L}_{1,p^{-1},\chi}=x)=\frac{1}{\prod_{i=1}^{\infty} (1-p^{-i})}\sum_{m=0}^\infty \exp(-\chi p^{m-x})\frac{(-1)^m p^{-{{m}\choose{2}}}}{\prod_{j=1}^m (1-p^{-j})}\text{ for any }x\in\mathbb{Z}.\]

Our next lemma provides an explicit formula for $\chi_0$ in certain special cases. For $p=2$, this lemma actually covers all $\mathbb{Z}$-valued random variables $\xi$ satisfying \eqref{xicond}.
\begin{lemma}\label{specialxi}
    Assume that there is an $\alpha\in(0,1)$ such that
    \[\mathbb{P}(\xi\equiv i\mod p)=\begin{cases}\alpha&\text{ for $i=0$},\\
    \frac{1-\alpha}{p-1}&\text{ for $i=1,2,\dots,p-1$.}
    \end{cases}\]
    
    Then
    \[\chi_0=\frac{p-1}p\prod_{i=1}^\infty \frac{(p-1)\beta_i}{p-\beta_i},\]
where \[\beta_i=1+(p-1)\left(\frac{p\alpha-1}{p-1}\right)^{i}.\]
\end{lemma}

\medskip

\textbf{Acknowledgments.} The author was supported by the Marie Sk\l{}odowska-Curie Postdoctoral Fellowship "RaCoCoLe". The author is grateful to Roger Van Peski for his comments.

\section{Preliminaries on the rescaled moment method}\label{sectionrescaled}

Given a partition $\lambda=(\lambda_1,\dots,\lambda_r)$, let \[G_\lambda=\bigoplus_{i=1}^r \mathbb{Z}/p^{\lambda_i}\mathbb{Z}\]
be the abelian $p$-group corresponding to $\lambda$, and let $\lambda'$ be the conjugate partition of $\lambda$. We also define $|\lambda|=\sum_{i=1}^r \lambda_i$.

We define $\MC(G_{\lambda'})$ as the number of chains of subgroups $\{0\}\subsetneq H_1\subsetneq \cdots\subsetneq H_{|\lambda|}= G_{\lambda'}$ of length $|\lambda|$, that is, of maximal length.  

Given a positive integer $d$, let
\begin{align*}\Sig_d&=\{(\lambda_1,\lambda_2,\dots,\lambda_d)\in \mathbb{Z}^d\,:\, \lambda_1\ge\lambda_2\cdots\ge \lambda_d\}\text{ and }\\
\Sig_d^{\ge 0}&=\{(\lambda_1,\lambda_2,\dots,\lambda_d)\in \mathbb{Z}^d\,:\, \lambda_1\ge\lambda_2\cdots\ge \lambda_d\ge 0\}.
\end{align*}

Note that $\Sig_d^{\ge 0}$ is just the set of partitions with at most $d$ parts.

For $\chi>0$, let $\mathcal{L}_{d,p^{-1},\chi}$ be a $\Sig_d\,$-valued random variable such that
\[\mathlarger{\mathbb{E} p^{\left\langle \mathcal{L}_{d,p^{-1},\chi}\,,\,\lambda \right\rangle}}=\frac{((p-1)\chi)^{|\lambda|}}{|\lambda|!}\MC(G_{\lambda'})\qquad \text{ for all }\lambda\in\Sig_d^{\ge 0}.\]
It was proved by Nguyen and Van Peski that such a random variable exists and its distribution is uniquely determined~\cite{nguyen2024rank}.

The convergence to $\mathcal{L}_{d,p^{-1},\chi}$ can be obtained using the so called rescaled moment method given in the next theorem, which is a new variant of the widely used moment method of Wood~\cite{wood2017distribution,wood2022probability}.

\begin{theorem}{\normalfont(Nguyen and Van Peski~\cite[Theorem 1.2 and Proposition 5.3.]{nguyen2024rank})} \label{rescaledmomentmethod}Fix a prime~$p$ and a positive integer $d$. Let $(\Gamma_j)_{j\ge 1}$ be a sequence of random
finitely-generated abelian $p$-groups and let $(n_j)_{j\ge 1}$ be a sequence of positive real numbers such that the following holds:
\begin{enumerate}[(a)]
\item The fractional part $\{-\log_p(n_j)\}$ converges to $\zeta$.
\item For all $\lambda\in \Sig_d^{\ge 0}$, we have
\[\lim_{j\to\infty} \mathbb{E}\
\frac{|\Hom(\Gamma_j,G_{\lambda'})|}{n_j^{|\lambda|}}=\frac{\chi_0^{|\lambda|}\MC(G_{\lambda'})}{|\lambda|!}.\]

\end{enumerate}

Then
\[\left(\rang(p^{i-1}\Gamma_j)-\lfloor\log_p(n_{j})+\zeta\rceil\right)_{i=1}^d\]
converges to $\mathcal{L}_{d,p^{-1},\chi}$ in distribution as $j\to\infty$, where $\chi=\frac{\chi_0 p^{-\zeta}}{p-1}$.
 
\end{theorem}

The quantities $\mathbb{E}|\Hom(\Gamma_j,G_{\lambda'})|$ are often referred to as the moments of $\Gamma_j$. If $\Gamma_n$ is defined as in the Introduction, then the following well-known lemma gives a useful expression for the moments of $\Gamma_n$, for the proof see for example \cite{meszaros2024universal}.

\begin{lemma}\label{momentexpression}
For any deterministic finite abelian $p$-group $G$, we have
\[\mathbb{E}|\Hom(\Gamma_n,G)|=\mathbb{E}|\Hom(\cok(L_n),G)|=\mathbb{E}|\{v\in G^n\,:\,L_nv=0\}|.\]
\end{lemma}

Combining Theorem~\ref{rescaledmomentmethod} and Lemma~\ref{momentexpression}, we see that Theorem~\ref{mainthm} follows once we prove the following theorem.

\begin{theorem}\label{thmmoment}
    For any finite abelian $p$-group $G$ of size $p^\ell$, we have
    \[\lim_{n\to\infty} n^{-\ell} \mathbb{E}|\{v\in G^n\,:\,L_nv=0\}|=\frac{\chi_0^\ell MC(G)}{\ell!}.\]
\end{theorem}

\section{Calculating the moments}

This section is devoted to the proof of Theorem~\ref{thmmoment}.

Let $H$ be a finite abelian $p$-group. Throughout this section, constants are allowed to depend on $H$ and $\xi$, but not on anything else.  

For a $v\in H^m$ and a character $\varrho\in \widehat{H}=\Hom(H,\mathbb{C}^*)$, we define
\[\|v\|_\varrho=|\{1\le i\le m\,:\,v_i\notin \ker \varrho\}|.\]

We also define
\[\tau(v)=\mathbb{P}\left(\sum_{i=1}^m \xi_i v_i=0\right),\]
where $\xi_1,\xi_2,\dots,\xi_m$ are i.i.d. copies of $\xi$.

We use $1$ to denote the trivial character in $\widehat{H}$.
\begin{lemma}\label{Fourierestimate}
There is a constant $c>0$, such that for any $m$ and any $v\in H^m$, if $g\in H$ is a random variable independent from $\xi_1,\dots,\xi_m$, then 
\begin{equation*}\left||H|\mathbb{P}\left(\sum_{i=1}^m \xi_i v_i=g\right)-1\right|\le \sum_{1\neq \varrho \in \widehat{H}} \exp\left(-c\|v\|_\varrho\right).\end{equation*}
In particular,
\[\Big||H|\tau(v)-1\Big|\le \sum_{1\neq \varrho \in \widehat{H}} \exp\left(-c\|v\|_\varrho\right).\]

\end{lemma}
\begin{proof}
Combining the inversion formula for the discrete Fourier transform and the fact that $g,\xi_1,\xi_2,\dots,\xi_m$ are independent, we obtain that
\begin{equation}\label{eqt1}
    |H|\mathbb{P}\left(\sum_{i=1}^m \xi_i v_i=g\right)=\sum_{\varrho\in \widehat{H}}\mathbb{E}\varrho\left(\sum_{i=1}^m \xi_i v_i-g\right)=1+\sum_{1\neq\varrho\in \widehat{H}}\left(\mathbb{E}\varrho(-g)\right)\prod_{i=1}^m\mathbb{E}\varrho\left( \xi_i v_i\right).
\end{equation}
By \cite[Lemma 2.2.]{wood2019random} and \eqref{xicond}, we have a $c>0$ such that 
\[\left|\mathbb{E}\varrho\left( \xi_i v_i\right)\right|\le \exp(-c)\]
whenever $v_i\notin\ker \varrho$, and in all other cases $\mathbb{E}\varrho\left( \xi_i v_i\right)=1$. Combining this with \eqref{eqt1}, the statement follows.
\end{proof}

Let $w\in H^h$. We assume that the components of $w$ generate $H$.

Let $X=(X_1,X_2,\dots,)\in H^{\mathbb{N}}$ be a random vector such that $X_i=w_i$ for all $1\le i\le h$ and $X_{h+1},X_{h+2},X_{h+3},\dots$ are independent uniform random elements of $H$.

For a vector $v$ in $H^m$ or $H^{\mathbb{N}}$, let $v_{\le i}$ be the prefix of $v$ of length $i$.

We define
\[Z_i=\prod_{j=h+1}^{h+i} |H| \tau(X_{\le j})\quad\text{ and }\quad Z_\infty=\prod_{j=h+1}^{\infty} |H| \tau(X_{\le j}).\]

\begin{lemma}\hfill\label{ZnLemma}
\begin{enumerate}[(a)]
    \item 
    There are constants $K=K_{H,\xi}<\infty$ and $\kappa_0=\kappa_{0,H,\xi}>1$ not depending on $w$ such that 
    \[\mathbb{E}\left(\sup_{i\ge 1} Z_i\right)^{\kappa}<K\]
    for all $1\le \kappa\le \kappa_0$.

    \item With probability $1$, $Z_\infty$ is well defined and finite. Moreover, $Z_i$ converge to $Z_\infty$ almost surely.

    \item Finally,
    \[\lim_{i\to \infty}\mathbb{E}Z_i=\mathbb{E}Z_\infty.\]
\end{enumerate}    
\end{lemma}
\begin{proof}
    Let us choose $\varepsilon>0$ small enough. Let $1\neq \varrho\in \widehat{H}$. We define
    \[T_\varrho=\sup\left(\left\{i\ge 1\,:\,\left\|X_{\le h+i}\right\|_\varrho\le \varepsilon i\right\}\cup \{0\}\right).\]
    Since $\mathbb{P}(X_{h+j}\notin \ker \varrho)\ge \frac{p-1}p$, it follows from the law of large numbers that $T_\varrho$ is almost surely finite provided that $\varepsilon$ is small enough.

    We define
    \[H_i=\bigcap_{\substack{1\neq \varrho\in \widehat{H}\\T_\varrho\ge i}} \ker\varrho,\]
    where the intersection over the empty set is defined to be $H$.

    By Lemma~\ref{Fourierestimate}, we have
    \begin{align}\label{estasp}
    |H|\tau(X_{\le h+i})&\le 1+\sum_{1\neq \varrho \in \widehat{H}} \exp\left(-c\|X_{\le h+i}\|_\varrho\right)\nonumber\\&\le \left(1+\sum_{\substack{1\neq \varrho \in \widehat{H}\\H_i\subseteq\ker \varrho}} \exp\left(-c\|X_{\le h+i}\|_\varrho\right)\right)\left(1+\sum_{\substack{1\neq \varrho \in \widehat{H}\\H_i\not\subseteq\ker \varrho}} \exp\left(-c\|X_{\le h+i}\|_\varrho\right)\right).
    \end{align}

    Note that if $H_i\not\subseteq\ker \varrho$, then $T_\varrho<i$, so $\|X_{\le h+i}\|_\varrho>\varepsilon i$. Therefore,
    \begin{equation}\label{estp1}
        1+\sum_{\substack{1\le \varrho \in \widehat{H}\\H_i\not\subseteq\ker \varrho}} \exp\left(-c\|X_{\le h+i}\|_\varrho\right)\le 1+|H|\exp(-c\varepsilon i)\le \exp\left(|H|\exp(-c\varepsilon i)\right).
    \end{equation}

    On the event $H_i=H$, we trivially have
    \begin{equation}\label{esttriv}1+\sum_{\substack{1\neq \varrho \in \widehat{H}\\H_i\subseteq\ker \varrho}} \exp\left(-c\|X_{\le h+i}\|_\varrho\right)=1.\end{equation}

    Now assume that $H_i\neq H$. In this case, there is a one-to-one correspondence between the sets $\{\varrho \in \widehat{H}\,:\,H_i\subseteq\ker \varrho\}$ and $\widehat{H/H_i}$. Thus, \begin{equation}\label{charcount}|\{\varrho \in \widehat{H}\,:\,H_i\subseteq\ker \varrho\}|=|H|/|H_i|.\end{equation}
    
    Since the components of $w$ generate $H$, we have a $j\in\{1,2,\dots,h\}$ such that $w_j\notin H_i$. Then, we can find a $\varrho\in \widehat{H}$ such that $H_i\subseteq \ker \varrho$ and $w_j\notin \ker \varrho$. For such a $\varrho$, we have  $\|X_{\le h+i}\|_\varrho\ge 1$. Combining this with \eqref{charcount}, we obtain that

    \begin{equation}\label{estnontriv}1+\sum_{\substack{1\neq \varrho \in \widehat{H}\\H_i\subseteq\ker \varrho}} \exp\left(-c\|X_{\le h+i}\|_\varrho\right)\le \frac{|H|}{|H_i|}-1+\exp(-c).
    \end{equation}

    Combining \eqref{estasp}, \eqref{estp1}, \eqref{esttriv} and \eqref{estnontriv}, we obtain that
    \begin{equation}\label{tauest}|H|\tau(X_{\le h+i})\le \left(\frac{|H|}{|H_i|}-\mathbbm{1}(H_i\neq H)(1-\exp(-c))\right) \exp\left(|H|\exp(-c\varepsilon i)\right).\end{equation}

    Introducing the notation $T_{\max}=\max_{1\neq \varrho\in \widehat{H}} T_\varrho$, it follows from \eqref{tauest} that
    \[\sup_{i\ge 1}Z_i\le K_0 \prod_{i=1}^{T_{\max}} \left(\frac{|H|}{|H_i|}-1+\exp(-c)\right),\]
    where $K_0$ is a constant defined as
    \[K_0=\prod_{i=1}^\infty \exp\left(|H|\exp(-c\varepsilon i)\right)=\exp\left(\frac{|H|\exp(-c\varepsilon )}{1-\exp(-c\varepsilon)}\right).\]

    If we choose $\kappa_0$ small enough then there is a constant $0<\gamma<1$ such that for all $1\le \kappa\le \kappa_0$ and for all subgroup $H'\neq H$ of $H$, we have
    \[\left(\frac{|H|}{|H'|}-1+\exp(-c)\right)^\kappa\le \gamma^2 \frac{|H|}{|H'|}. \]

    Thus, for all $1\le \kappa\le \kappa_0$, we have
    \begin{equation}\label{supest}\left(\sup_{i\ge 1}Z_i\right)^\kappa\le K_0^{\kappa_0} \gamma^{2T_{\max}} \prod_{i=1}^{T_{\max}} \frac{|H|}{|H_i|}.
    \end{equation}
    Note that the choice of $\gamma$ only depends on $H$ and $\xi$, but not on $\varepsilon$.

    Let $t$ be a positive integer and let $C_1\subseteq C_2\subseteq\cdots\subseteq C_t\neq H$ be a chain of subgroups of $H$.

    \begin{lemma}\label{lemmasupZgivenchain}
        Let $\mathcal{A}$ be the event that $T_{\max}=t$ and $H_i=C_i$ for all $1\le i\le t$. Provided  that $\varepsilon$ is small enough, we have
        \[\mathbb{P}(\mathcal{A})\le \gamma^{-t}\prod_{i=1}^t \frac{|C_i|}{|H|}.\]
        Moreover, for all $1\le \kappa\le \kappa_0$, we have
        \[\mathbb{E}\mathbbm{1}(\mathcal{A}) \left(\sup_{i\ge 1} Z_i\right)^\kappa\le \gamma^{t}K_0^{\kappa_0}.\]
    \end{lemma}
    \begin{proof}
        For $1\neq \varrho\in \widehat{H}$, let
        \[B_{\varrho}=\{1\le j\le T_\varrho\,:\,X_{h+j}\notin \ker \varrho \}.\]
         On the event $\mathcal{A}$, if $X_{h+j}\notin C_j$ for some $1\le j\le t$, then there is a $1\neq \varrho\in \widehat{H}$ such that $T_\varrho\ge j$ and $X_{h+j}\notin \ker \varrho$. In particular, $j\in B_\varrho$.

        Therefore, we just proved that on the event $\mathcal{A}$, we have \[X_{h+j}\in C_j\text{ for all }j\in \{1,2,\dots,t\}\setminus \bigcup_{1\neq\varrho \in \widehat{H}} B_\varrho.\] Observe that on the event $\mathcal{A}$, we have
        \[\left|\bigcup_{1\neq\varrho \in \widehat{H}} B_\varrho\right|\le \sum_{1\neq\varrho \in \widehat{H}} \varepsilon T_\varrho\le \varepsilon|H|t. \]

        Note that assuming that $\varepsilon$ small enough the right hand side is less than $t$.
        
        Below, $\sum_B$ will stand for a summation over all $\lfloor  \varepsilon|H|t \rfloor$ element subsets $B$ of $\{1,2,\dots,t\}$.

        It follows from the discussion above that
        \begin{align*}
            \mathbb{P}&(T_{\max}=t\text{ and }H_i=C_i\text{ for all }1\le i\le t)\\&\le \sum_B \mathbb{P}(X_{h+j}\in C_j\text{ for all }j\in \{1,\dots,t\}\setminus B)\\&
            = \sum_B \prod_{j\in \{1,\dots,t\}\setminus B} \frac{|C_j|}{|H|}\\&
            \le {{t}\choose{ \lfloor\varepsilon|H|t\rfloor}} |H|^ {\varepsilon|H|t} \prod_{j=1}^t \frac{|C_j|}{|H|}. 
    \end{align*}
By choosing $\varepsilon$ small enough we can achieve that ${{t}\choose{ \lfloor\varepsilon|H|t\rfloor}} |H|^ {\varepsilon|H|t}\le \gamma^{-t}$. Thus, the first statement follows. The second statement follows by combining the first one with~\eqref{supest}.
\end{proof}

It is easy to see that there is a constant $c_2$ such that the number of chains of subgroups $C_1\subseteq C_2\subseteq\cdots\subseteq C_t$ is at most $c_2(t+1)^{\log_p|H|}$. Combining this estimate with Lemma~\ref{lemmasupZgivenchain}, we get that 
\[\mathbb{E}\mathbbm{1}(T_{\max}=t)\left(\sup_{i\ge 1} Z_i\right)^\kappa\le c_2 (t+1)^{\log_p|H|}\gamma^{t}K_0^{\kappa_0}.\]

Note that this estimate also true for $t=0$. Thus,
\[\mathbb{E}\left(\sup_{i\ge 1} Z_i\right)^\kappa=\sum_{t=0}^\infty\mathbb{E}\mathbbm{1}(T_{\max}=t)\left(\sup_{i\ge 1} Z_i\right)^\kappa\le\sum_{t=0}^\infty c_2 (t+1)^{\log_p|H|}\gamma^{t}K_0^{\kappa_0}<\infty.\]
Therefore, part (a) of the lemma follows.

Let $i_0$ be the smallest $i$ such that $1-|H|\exp(-c\varepsilon i)>0$.

Assuming that $\ell> i> \max(T_{\max},i_0)$, by Lemma~\ref{Fourierestimate}, we have
\[\prod_{j=i+1}^\ell(1-|H|\exp(-c\varepsilon j))\le \frac{Z_\ell}{Z_i}\le \prod_{j=i+1}^\ell (1+|H|\exp(-c\varepsilon j)).\]
Thus, it follows easily that with probability $1$, $Z_\infty$ is defined and finite. Moreover, $Z_i$ converges to $Z_\infty$ almost surely. Thus, part (b) follows.

    Finally, by parts (a) and (b), we can apply the dominated convergence theorem to obtain that
    \[\lim_{i\to \infty}\mathbb{E}Z_i=\mathbb{E}Z_\infty.\qedhere\]
   \end{proof}

Let $G$ be a finite abelian $p$-group. Let $0\le n_1<n_2<\cdots<n_{k+1}$ be a sequence of integers, and let $g_1,g_2,\dots,g_{k}\in G$. We define $G_0$ as $\{0\}$ and $G_i$ as the subgroup of $G$ generated by $g_1,g_2,\dots,g_i$. Let $\V{n_1,\dots,n_{k+1}}{g_1,\dots,g_{k}}$ be the set of vectors $v\in G^{n_{k+1}}$ satisfying
\begin{itemize}
\item $v_{n_i+1}=g_i$ for all $i=1,2,\dots,k$;
\item $v_j\in G_i$ for all $i=0,1,\dots,k$ and $n_{i}+1<j\le n_{i+1}$.
\end{itemize}
Here we defined $n_0$ to be $-1$.

Note that
\[\left|\V{n_1,\dots,n_{k+1}}{g_1,\dots,g_{k}}\right|=\prod_{i=0}^{k}|G_i|^{n_{i+1}-n_i-1}.\]

Let $X=(X_1,\dots,X_{n_{k+1}})$ be a uniform random element of $\V{n_1,\dots,n_{k+1}}{g_1,\dots,g_{k}}$. Let
\[R=\left(\prod_{i=0}^{k}|G_i|^{n_{i+1}-n_i-1}\right)\prod_{i=1}^{n_{k+1}}\tau(X_{\le i}).\]

Note that for any $v\in G^n$, we have
\begin{equation}\label{randomvscounting0}\mathbb{P}(L_nv=0)=\prod_{i=1}^n \tau(v_{\le i}).\end{equation}
Therefore,
\begin{equation}\label{randomvscounting}\mathbb{E}R=\mathbb{E}\left|\{v\in \V{n_1,\dots,n_{k+1}}{g_1,\dots,g_{k}}\,:\,L_{n_{k+1}}v=0\}\right|.\end{equation}


Let 
\[M=\max_{H\le G} K_{H,\xi}\quad\text{ and }\quad\kappa_1= \min_{H\le G} \kappa_{0,H,\xi},\]
where the constants $K_{H,\xi}$ and $\kappa_{0,H,\xi}$ are provided by Lemma~\ref{ZnLemma}.
\begin{lemma}\label{RLemma}
For any $1\le\kappa\le \kappa_1$, we have
\[\mathbb{E}R^{\kappa}\le M^{k}.\]
\end{lemma}
\begin{proof}
    We prove by induction on $k$. For $k=0$, the statement is trivial. Assume that $k>0$. Let
    \[R_0=\prod_{j=n_{k}+2}^{n_{k+1}} |G_{k}|\tau(X_{\le j}).\]

    Let $v\in \V{n_1,\dots,n_{k}}{g_1,\dots,g_{k-1}}$. We set $h=n_{k}+1$, $H=G_{k}$. Let $w\in H^h$ be defined by appending $v$ with $g_{k}$. Let $Z_n$ be defined as in Lemma~\ref{ZnLemma}.

    Note that $R_0$ conditioned on the event that $X_{\le n_{k}}=v$ has the same distribution as $Z_{n_{k+1}-n_{k}-1}$.

    Then
    \begin{align*}\mathbb{E}[R^\kappa\,|\,X_{\le n_{k}}=v]&=\left(\left(\prod_{i=0}^{k-1}|G_i|^{n_{i+1}-n_i-1}\right)\prod_{i=1}^{n_{k}+1}\tau(v_{\le i})\right)^{\kappa} \mathbb{E}[R_0^{\kappa}\,|\,X_{\le n_{k}}=v]\\&\le \left(\left(\prod_{i=0}^{k-1}|G_i|^{n_{i+1}-n_i-1}\right)\prod_{i=1}^{n_{k}}\tau(v_{\le i})\right)^{\kappa} \mathbb{E}Z_{n_{k+1}-n_{k}-1}^\kappa\\&
    \le M\left(\left(\prod_{i=0}^{k-1}|G_i|^{n_{i+1}-n_i-1}\right)\prod_{i=1}^{n_{k}}\tau(v_{\le i})\right)^{\kappa},\end{align*}
    where the last step we used Lemma~\ref{ZnLemma}.

    Combining this with the induction hypothesis, we obtain that
    \[\mathbb{E}R^{\kappa}\le M \mathbb{E}\left(\left(\prod_{i=0}^{k-1}|G_i|^{n_{i+1}-n_i-1}\right)\prod_{i=1}^{n_{k}}\tau(X_{\le i})\right)^{\kappa}\le M^k.\qedhere\]
\end{proof}

Let 
\[A_n=\left\{v\in \mathbb{F}_p^n\,:\,v_1\neq 0, |\{1\le i
\le n\,: v_i\neq 0\}|\ge \frac{n}3\right\}.\]

The following lemma is clearly stronger than Lemma~\ref{lemmachi0exists}.
\begin{lemma}\label{lemmachi0existsS}\hfill
\begin{enumerate}[(a)]
    \item The limit
    \[\chi_0=\lim_{n\to\infty} \mathbb{E}|\{v\in\mathbb{F}_p^n\,:\,v_1\neq 0, L_nv=0\}|\]
    exists. 
    \item We have $0<\chi_0<\infty$.
    \item Finally,
    \[\lim_{n\to\infty} \mathbb{E}|\{v\in A_n\,:\, L_nv=0\}|=\chi_0.\]
    \end{enumerate}
\end{lemma}
\begin{proof}
Part (a): Let $0\neq g\in \mathbb{F}_p$. Combining equation~\eqref{randomvscounting0} and Lemma~\ref{ZnLemma} with the choice $H=\mathbb{F}_p$, $h=1$, $w_1=g$, we obtain that,
\[\lim_{n\to\infty} \mathbb{E}|\{v\in\mathbb{F}_p^n\,:\,v_1= g, L_nv=0\}|=\lim_{n\to\infty} \tau(w)\mathbb{E}Z_{n-1}\]
    exists and finite. Summing this over all the possible choices of $0\neq g\in \mathbb{F}_p$, we get the statement.

Part (b): We already seen that $\chi_0<\infty$. To show that $\chi_0>0$, let $0\neq g\in \mathbb{F}_p$. We will rely on the proof of Lemma~\ref{ZnLemma} with the choice $H=\mathbb{F}_p$, $h=1$, $w_1=g$. Let $X$ be defined as in that lemma. Let $q$ be the probability that $\xi$ is divisible by $p$. By \eqref{xicond}, we have $q>0$. Note that
\[\tau(X_{\le i})\ge q^i.\]

Choose a positive integer $t$ such that $1-p\exp(-c\varepsilon t)>0$ and $\mathbb{P}(T_{\max}\le t)>0$. Then following the argument of the proof of Lemma~\ref{ZnLemma}, we obtain that

\[\mathbb{E} \mathbbm{1}(T_{\max}\le t) Z_\infty\quad{\ge}\quad\mathbb{P}(T_{\max}\le t) \left(\prod_{i=1}^{t}pq^{i+1}\right) \prod_{i=t+1}^\infty \left(1-p\exp(-c\varepsilon i)\right) >0.\]

Since by Lemma~\ref{ZnLemma} and \eqref{randomvscounting0}, we have
\[\lim_{n\to\infty}\mathbb{E}|\{v\in \mathbb{F}_p^n\,:\, v_1=g,L_nv=0\}|=q\mathbb{E}Z_\infty,\]
 part (b) follows.

Part (c): let $0\neq g\in \mathbb{F}_p$. We apply Lemma~\ref{RLemma} with the choice $G=\mathbb{F}_p$, $k=1$, $g_1=g$, $n_1=0$ and $n_2=n$. Let $X$ be defined as in that lemma. Then combining Lemma~\ref{RLemma} with~\eqref{randomvscounting0}, we have
\[\mathbb{E} \left(p^{n-1} \mathbb{P}(L_nX=0)\right)^{\kappa_1}\le M.\]

It follows from the law of large numbers that
\[\lim_{n\to \infty}\mathbb{P}(X\notin A_n)=0.\]

Choose $\kappa_2$ such that $\frac{1}{\kappa_1}+\frac{1}{\kappa_2}=1$. By H\"{o}lder's inequality,
\begin{align*}\lim_{n\to\infty}&\mathbb{E} |\{v\in\mathbb{F}_p^n\setminus A_n, v_1=g,L_nv=0\}|\\&=\lim_{n\to\infty}\mathbb{E} \mathbbm{1}(X\notin A_n)p^{n-1} \mathbb{P}(L_nX=0)&\\&\le \lim_{n\to\infty}\mathbb{P}(X\notin A_n)^{1/\kappa_2}\left(\mathbb{E} \left(p^{n-1} \mathbb{P}(L_nX=0)\right)^{\kappa_1}\right)^{1/\kappa_1}\\&\le \lim_{n\to\infty}\mathbb{P}(X\notin A_n)^{1/\kappa_2}  M^{1/\kappa_1}\\&=0.\end{align*}

Summing these over $0\neq g\in \mathbb{F}_p$ and using part (a), the statement follows.
\end{proof}

Let us consider $n_1,n_2,\dots,n_{k+1},g_1,g_2,\dots, g_{k}$ and $G_0,G_1,\dots,G_{k}$ as before. We say that $(n_1,n_2,\dots,n_{k+1})$ and $(g_1,g_2,\dots, g_{k})$ form an $(n,k)$-skeleton, if $n_{k+1}=n$ and $G_0, G_1,\dots, G_{k}$ are pairwise distinct. Moreover, by an $n$-skeleton we mean an $(n,k)$-skeleton for some $k$. The proof of the following lemma is straightforward.

\begin{lemma}
Let $v\in G^n$, then there is a unique $n$-skeleton $(n_1,\dots,n_{k+1}),(g_1,\dots, g_{k})$ such that $v\in \V{n_1,\dots,n_{k+1}}{g_1,\dots, g_{k}}$.
\end{lemma}
 We call the $n$-skeleton provided by the previous lemma the skeleton of $v$, and we use $k(v), n_1(v),n
 _2(v),\dots$ to denote the $k,n_1,n_2,\dots$ above, respectively. 

 Let $\ell=\log_p|G|$. Note that $k(v)\le \ell$.

\begin{lemma}\label{lemmaneg1}
We have
\[\lim_{n\to\infty}n^{-\ell}\mathbb{E}|\{v\in G^n\,:\,k(v)<\ell,\, L_nv=0\}|=0.\]
\end{lemma}
\begin{proof}
The number of $(n,k)$-skeletons is at most $(n|G|)^k$. Combining this with Lemma~\ref{RLemma} and~\eqref{randomvscounting}, we see that for any $k$, we have
\[\mathbb{E}|\{v\in G^n\,:\,k(v)=k,\, L_nv=0\}|\le (n|G|M)^k.\]
Thus, the statement follows easily.
\end{proof}

From now on we mostly restrict our attention to $(n,\ell)$-skeletons. Note that for an $(n,\ell)$-skeleton $(n_1,\dots,n_{\ell+1}),(g_1,\dots, g_{\ell})$, the chain $G_0\subsetneq G_1\subsetneq\cdots\subsetneq G_\ell$ is a maximal chain of subgroups in $G$. Thus, $G_\ell=G$ and $G_{i}/G_{i-1}\cong \mathbb{F}_p$ for all $i=1,2,\dots,\ell$.

Let $(n_1,\dots,n_{\ell+1}),(g_1,\dots, g_{\ell})$ be an $(n,\ell)$-skeleton. We say that this skeleton is well separated if \begin{equation}\label{wellsepcond}n_{i+1}-n_{i}>3\log^2 n\quad\text{ for all }i=1,\dots,\ell.\end{equation} 

\begin{lemma}\label{lemmaneq2}
We have
\[\lim_{n\to\infty}n^{-\ell}\mathbb{E}|\{v\in G^n\,:\,k(v)=\ell,\, L_nv=0,\text{ the skeleton of  }v\text{ is not well separated}\}|=0.\]
\end{lemma}
\begin{proof}
It is easy to prove that the number of $(n,\ell)$-skeletons which are not well separated is at most $3\ell |G|^\ell n^{\ell-1}\log^2 n$. Thus, the lemma follows the same way as Lemma~\ref{lemmaneg1}.
\end{proof}

Let $(n_1,\dots,n_{\ell+1}),(g_1,\dots, g_{\ell})$ be a well separated skeleton. We partition $\V{n_1,\dots,n_{\ell+1}}{g_1,\dots, g_{\ell}}$ into two subsets $\U{n_1,\dots,n_{\ell+1}}{g_1,\dots, g_{\ell}}$ and $\W{n_1,\dots,n_{\ell+1}}{g_1,\dots, g_{\ell}}$ such that  $\U{n_1,\dots,n_{\ell+1}}{g_1,\dots, g_{\ell}}$ consists of those elements $v$ of $\V{n_1,\dots,n_{\ell+1}}{g_1,\dots, g_{\ell}}$ which satisfy that
\begin{equation*}|\{n_{i}+1\le j\le n_{i+1}\,:\,v_j\notin G_{i-1}\}|\ge\frac{n_{i+1}-n_i}3\quad \text{ for all }i=1,2,\dots,\ell.\end{equation*}

In the lemma below $\sum_{ws(n)}$ denotes a sum over all well separated $(n,\ell)$-skeletons.

\begin{lemma}\label{lemmaneq3}
We have
\[\lim_{n\to\infty}n^{-\ell}\sum_{ws(n)}\mathbb{E}\left|\left\{v\in \W{n_1,\dots,n_{\ell+1}}{g_1,\dots, g_{\ell}}\,:\, L_nv=0\right\}\right|=0.\]
\end{lemma}

\begin{proof}
Let $(n_1,\dots,n_{\ell+1}),(g_1,\dots, g_{\ell})$ be a well separated $(n,\ell)$-skeleton, and let $X$ be a uniform random element of $\V{n_1,\dots,n_{\ell+1}}{g_1,\dots, g_{\ell}}$, and define $R$ as in Lemma~\ref{RLemma}. Let $\varepsilon>0$. It follows from the law of large numbers that provided that $n$ is large enough, we have
\[\mathbb{P}\left(X\in \W{n_1,\dots,n_{\ell+1}}{g_1,\dots, g_{\ell}}\right)<\varepsilon\]
for all choices of well separated $(n,\ell)$-skeletons.

Choose $\kappa_2$ such that $\frac{1}{\kappa_1}+\frac{1}{\kappa_2}=1$. Combining \eqref{randomvscounting0} with H\"{o}lder's inequality and Lemma~\ref{RLemma}, we obtain that
\begin{align*}
\mathbb{E}&\left|\left\{v\in \W{n_1,\dots,n_{\ell+1}}{g_1,\dots, g_{\ell}}\,:\, L_nv=0\right\}\right|\\&=\mathbb{E} \mathbbm{1}\left(X\in \W{n_1,\dots,n_{\ell+1}}{g_1,\dots, g_{\ell}}\right) R\\&
\le \left(\mathbb{E} R^{\kappa_1}\right)^{1/\kappa_1} \left(\mathbb{P}\left(X\in \W{n_1,\dots,n_{\ell+1}}{g_1,\dots, g_{\ell}}\right)\right)^{1/\kappa_2}\\&\le M^{\ell/\kappa_1} \varepsilon^{1/\kappa_2}.
\end{align*}
Since the number of $(n,\ell)$-skeletons is at most $(|G|n)^\ell$, and $\varepsilon$ can be arbitrary the statement follows.\end{proof}

Let $(n_1,\dots,n_{\ell+1}),(g_1,\dots, g_{\ell})$ be a well separated $(n,\ell)$-skeleton, and let 
$v\in \V{n_1,\dots,n_{\ell+1}}{g_1,\dots, g_{\ell}}$. For $i=1,\dots,\ell$, we define
\[v^{(i)}=(v_{n_i+1}+G_{i-1},\,v_{n_i+2}+G_{i-1},\dots,\,v_{n_{i+1}}+G_{i-1})\in \left(G_i/G_{i-1}\right)^{n_{i+1}-n_i}.\]

\begin{lemma}\label{LemmaforU}
Assume that $n$ is large enough. Let $(n_1,\dots,n_{\ell+1}),(g_1,\dots, g_{\ell})$ be a well separated $(n,\ell)$-skeleton and let $v\in \U{n_1,\dots,n_{\ell+1}}{g_1,\dots, g_{\ell}}$. 

Then for $i=2,3,\dots, \ell$ and $1\le j\le n_{i+1}-n_{i}$, we have
\[(1-(p-1)\exp(-c\log^2 n))^{i-1}\frac{\tau(v^{(i)}_{\le j})}{p^{i-1}}\le\tau(v_{\le n_{i}+j})\le (1+(p-1)\exp(-c\log^2 n))^{i-1} \frac{\tau(v^{(i)}_{\le j})}{p^{i-1}} .\]
\end{lemma}
\begin{proof}
    Obviously $\sum_{k=1}^{n_i+j} \xi_k v_k=\sum_{k=n_1+1}^{n_i+j} \xi_k v_k$.
    
    Note that $\sum_{k=n_1+1}^{n_i} \xi_k v_k\in G_{i-1}$. Thus, on the event that $\sum_{k=n_1+1}^{n_i+j} \xi_k v_k=0$, we must have $\sum_{k=n_i+1}^{n_i+j} \xi_k v_k\in G_{i-1}$. Therefore,
    \[\mathbb{P}\left(\sum_{k=n_1+1}^{n_i+j} \xi_k v_k=0\right)=\mathbb{P}\left(\sum_{k=n_i+1}^{n_i+j} \xi_k v_k\in G_{i-1}\right)\mathbb{P}\left(\sum_{k=n_1+1}^{n_i+j} \xi_k v_k=0\,\Big|\,\sum_{k=n_i+1}^{n_i+j} \xi_k v_k\in G_{i-1}\right).\]
    Continuing in a similar manner
    \begin{align*}
    \mathbb{P}&\left(\sum_{k=n_1+1}^{n_i+j} \xi_k v_k=0\right)\\&=\mathbb{P}\left(\sum_{k=n_i+1}^{n_i+j} \xi_k v_k\in G_{i-1}\right) \prod_{h=1}^{i-1} \mathbb{P}\left(\sum_{k=n_h+1}^{n_i+j} \xi_k v_k\in G_{h-1}\,\Big|\,\sum_{k=n_{h+1}+1}^{n_{i}+j} \xi_k v_k\in G_{h}\right)
    \\&=\mathbb{P}\left(\sum_{k=n_i+1}^{n_i+j} \xi_k (v_k+G_{i-1})=0+G_{i-1}\right) \\&\qquad\cdot\prod_{h=1}^{i-1} \mathbb{P}\left(\sum_{k=n_h+1}^{n_{h+1}} \xi_k (v_k+G_{h-1})=-\sum_{k=n_{h+1}+1}^{n_{i}+j} \xi_k (v_k+G_{h-1})\,\Big|\,\sum_{k=n_{h+1}+1}^{n_{i}+j} \xi_k v_k\in G_{h}\right).
\end{align*}

Here
\[\mathbb{P}\left(\sum_{k=n_i+1}^{n_i+j} \xi_k (v_k+G_{i-1})=0+G_{i-1}\right)=\tau(v^{(i)}_{\le j}).\]
Let $h\in\{1,2,\dots,i-1\}$. Since $(n_1,\dots,n_{\ell+1}),(g_1,\dots, g_{\ell})$ is a well separated $(n,\ell)$-skeleton and $v\in \U{n_1,\dots,n_{\ell+1}}{g_1,\dots, g_{\ell}}$, we see that
\[\|v^{(h)}\|_\varrho>\log^2 n\quad\text{ for all }1\neq \varrho\in \widehat{G_h/G_{h-1}}.\]
Thus, by Lemma~\ref{Fourierestimate}, we obtain that
\begin{multline*}
    \frac{1}p\left(1-(p-1)\exp(-c\log^2 n)\right)\\\le \mathbb{P}\left(\sum_{k=n_h+1}^{n_{h+1}} \xi_k (v_k+G_{h-1})=-\sum_{k=n_{h+1}+1}^{n_{i}+j} \xi_k (v_k+G_{h-1})\,\Big|\,\sum_{k=n_{h+1}+1}^{n_{i}+j} \xi_k v_k\in G_{h}\right)\\\le \frac{1}p\left(1+(p-1)\exp(-c\log^2 n)\right).
\end{multline*}
Therefore, the statement follows.
\end{proof}

\begin{lemma}\label{probdecompose}
We have
\[\lim_{n\to\infty} \max \left|\frac{\prod_{i=1}^\ell p^{-(i-1)(n_{i+1}(v)-n_i(v))}\mathbb{P}(L_{n_{i+1}(v)-n_i(v)} v^{(i)}=0)}{\mathbb{P}(L_nv=0)}-1\right|=0,\]
where the max is over all $v$ such that $v\in\U{n_1,\dots,n_{\ell+1}}{g_1,\dots, g_{\ell}}$ for some well separated $(n,\ell)$-skeleton $(n_1,\dots,n_{\ell+1}),(g_1,\dots, g_{\ell})$.

\end{lemma}
\begin{proof}
Let  $v\in\U{n_1,\dots,n_{\ell+1}}{g_1,\dots, g_{\ell}}$ for some well separated $(n,\ell)$-skeleton $(n_1,\dots,n_{\ell+1}),(g_1,\dots, g_{\ell})$. By Lemma~\ref{LemmaforU},  we have
\begin{align*}\mathbb{P}(L_n v)&=\prod_{i=1}^\ell\prod_{j=1}^{n_{i+1}-n_i}\tau(v_{\le n_i+j})\\&\le \prod_{i=1}^\ell p^{-(i-1)(n_{i+1}-n_i)}\mathbb{P}(L_{n_{i+1}-n_i}v^{(i)}=0)(1+(p-1)\exp(-c\log^2 n))^{(i-1)(n_{i+1}-n_i)}\\&\le (1+(p-1)\exp(-c\log^2 n))^{\ell n}\prod_{i=1}^\ell p^{-(i-1)(n_{i+1}-n_i)}\mathbb{P}(L_{n_{i+1}-n_i}v^{(i)}=0).
\end{align*}
Note that here $\lim_{n\to\infty} (1+(p-1)\exp(-c\log^2 n))^{\ell n}=1.$ We can obtain a matching lower bound in a similar manner.
\end{proof}

Consider a sequence $0\le n_1<\dots<n_{\ell+1}$ satisfying \eqref{wellsepcond}. For a maximal chain of subgroups $\{0\}=G_0\subsetneq G_1\subsetneq \cdots\subsetneq G_\ell=G$, let
$\U{n_1,\dots,n_{\ell+1}}{G_1,\dots,G_\ell}=\bigcup \U{n_1,\dots,n_{\ell+1}}{g_1,\dots,g_\ell}$
where the union is over all the choices of $g_i=G_i\setminus G_{i-1}$.

\begin{lemma}\label{lemmanonneg}
We have
\[\lim_{n\to\infty}\max \left|\mathbb{E}\left|\left\{v\in \U{n_1,\dots,n_{\ell+1}}{G_1,\dots,G_\ell}\,:\,L_nv=0\right\}\right|-\chi_0^\ell\right|=0,\]
where the maximum is over all the choices $0\le n_1<\dots<n_{\ell+1}=n$ satisfying \eqref{wellsepcond} and maximal chains of subgroups $\{0\}=G_0\subsetneq G_1\subsetneq \cdots\subsetneq G_\ell=G$.

\end{lemma}
\begin{proof}
 Consider a sequence $0\le n_1<\dots<n_{\ell+1}=n$ satisfying \eqref{wellsepcond} and a maximal chain of subgroups $\{0\}=G_0\subsetneq G_1\subsetneq \cdots\subsetneq G_\ell=G$. Let us consider the map
 \[v\mapsto (v^{(1)},\dots,v^{(\ell)})\]
 restricted to $\U{n_1,\dots,n_{\ell+1}}{G_1,\dots,G_\ell}$. After fixing an isomorphism between $\mathbb{F}_p$ and $G_i/G_{i-1}$, the image of this map is $A_{n_2-n_1}\times A_{n_3-n_2}\times \cdots\times A_{n_{\ell+1}-n_\ell}$. Moreover, the preimage of each element of this set has size $\prod_{i=1}^\ell p^{(i-1)(n_{i+1}-n_i)}$. Combining this observation with Lemma~\ref{probdecompose}, we have
\[
 \lim_{n\to\infty}\max \left|\frac{\mathbb{E}|\{v\in \U{n_1,\dots,n_{\ell+1}}{G_1,\dots,G_\ell}\,:\,L_nv=0\}|}{\prod_{i=1}^{\ell} \mathbb{E}|\{v\in A_{n_{i+1}-n_i}\,:\,L_{n_{i+1}-n_i}v=0\}|}-1\right|=0,\]
where the maximum is over all the choices $0\le n_1<\dots<n_{\ell+1}=n$ satisfying \eqref{wellsepcond} and a maximal chain of subgroups $\{0\}=G_0\subsetneq G_1\subsetneq \cdots\subsetneq G_\ell=G$.
Combining this with part (c) of Lemma~\ref{lemmachi0existsS}, the statement follows. 
\end{proof}

The number of tuples $0\le n_1<n_2<\cdots<n_{\ell+1}=n$ satisfying \eqref{wellsepcond} is $(1+o(1))\frac{n^\ell}{\ell!}$. Thus, Theorem~\ref{thmmoment} follows by combining Lemma~\ref{lemmaneg1}, Lemma~\ref{lemmaneq2}, Lemma~\ref{lemmaneq3} and Lemma~\ref{lemmanonneg}.

\section{Values of $\chi_0$ in special cases}
In this section, we prove Lemma~\ref{specialxi}.

Let $1\neq\varrho\in \widehat{\mathbb{F}_p}$. If $0\neq g\in \mathbb{F}_p$, then
\[\mathbb{E}\varrho(\xi g)=\frac{p\alpha-1}{p-1}\varrho(0)+\sum_{h\in \mathbb{F}_p}\frac{1-\alpha}{p-1}\varrho(h)=\frac{p\alpha-1}{p-1}.\]
Given a vector $v\in \mathbb{F}_p^m$, let
\[\|v\|=|\{1\le i\le m\,:\,v_i\neq 0\}|.\]
Then it follows from \eqref{eqt1} that,
\begin{equation}\label{ptauformula}p\tau(v)=1+(p-1)\left(\frac{p\alpha-1}{p-1}\right)^{\|v\|}=\beta_{\|v\|}.\end{equation}
Let us define $X=(X_1,X_2,\dots)$, where $X_1,X_2,\dots$ are independent, $X_1$ is a uniform random element of $\mathbb{F}_p\setminus \{0\}$, and $X_2,X_3,\dots$ are uniform random elements of $\mathbb{F}_p$. 

Let \[Y_n=\prod_{i=1}^n p\tau(X_{\le i})\quad\text{ and }\quad Y_\infty=\prod_{i=1}^n p\tau(X_{\le i}).\]

Then
\[\mathbb{E}|\{v\in \mathbb{F}_p^n\,:\,v_1\neq 0, L_nv=0\}|=\frac{p-1}p\mathbb{E} Y_n.\]
By conditioning on $X_1$ and applying Lemma~\ref{ZnLemma}, we obtain
\begin{equation}\label{Yinf}\lim_{n\to\infty}\mathbb{E}|\{v\in \mathbb{F}_p^n\,:\,v_1\neq 0, L_nv=0\}|=\lim_{n\to\infty}\frac{p-1}p\mathbb{E} Y_n=\frac{p-1}p\mathbb{E} Y_\infty.\end{equation}

For $i\ge 1$, let us define $T_i=\min\{j\,:\,\|X_{\le j}\|=i\}$. By \eqref{ptauformula}, we see that
\[p \tau(X_{\le j})=\beta_i \quad \text{ for all }T_i\le j<T_{i+1}.\]

Thus,
\[Y_\infty=\prod_{i=1}^\infty \beta_i^{T_{i+1}-T_i}.\]

Note that $T_2-T_1,T_3-T_2,T_4-T_3,\dots$ are i.i.d. random variables, such that for all $k\ge 1$, we have
\[\mathbb{P}(T_{i+1}-T_i=k)=\frac{p-1}{p^k}.\]

Thus, 
\[\mathbb{E}\beta_i^{T_{i+1}-T_i}=\sum_{k=1}^\infty (p-1)\left(\frac{\beta_i}p\right)^k=\frac{(p-1)\beta_i}{p-\beta_i}.\]

By the independence of $T_2-T_1,T_3-T_2,T_4-T_3,\dots$, we obtain
\[\mathbb{E} Y_\infty=\prod_{i=1}^\infty \frac{(p-1)\beta_i}{p-\beta_i}.\]
Combining this with \eqref{Yinf}, Lemma~\ref{specialxi} follows.

\bibliography{references}

\bibliographystyle{plain}

{\tt meszaros@renyi.hu}

HUN-REN Alfr\'ed R\'enyi Institute of Mathematics,

Budapest, Hungary

\end{document}